\documentclass[10pt,twoside,final]{siamltex}

\usepackage{amssymb}
\usepackage{amsmath,xcolor}
\usepackage{mathrsfs}

\newcommand{\varep}{\varepsilon}

\newcommand{\norm} [1]{\left\| {#1}\right\|}

\newcommand{\zz} {\mathbf {z}}
\newcommand{\vv} {\mathbf {v}}
\newcommand{\uu} {\mathbf {u}}
\newcommand{\s} {\mathbf {s}}

\newcommand{\ddt}{\frac d{dt} }

\def\beq{\begin{equation}}
\def\eeq{\end{equation}} 

\def\beqs{\begin{equation*}}
\def\eeqs{\end{equation*}}

\def\bals{\begin{align*}}
\def\eals{\end{align*}}

\def\bspl{\begin{split}}
\def\espl{\end{split}}

\def\myclearpage{}

\title{Analysis of expanded mixed finite element methods for the generalized Forchheimer equations}

\author{
 Thinh T. Kieu\thanks{Department of Mathematics, University of North Georgia, Gainesville Campus, 3820 Mundy Mill Rd., Oakwood, GA 30566, U.S.A. ({\tt thinh.kieu@ung.edu}).}}

\begin{document}

\maketitle
            
\begin{abstract} The nonlinear Forchheimer equations are used to describe the dynamics of fluid flows in porous media when Darcy's law is not applicable. In this article, we consider the generalized Forchheimer flows for slightly compressible fluids, and then study the expanded mixed finite element method applied to the initial boundary value problem for the resulting degenerate parabolic equation for pressure. The bounds for the solutions, time derivative and gradient of solutions are established. Utilizing the monotonicity properties of Forchheimer equation and boundedness of solutions, a {\it priori } error estimates for solution are obtained in $L^2$-norm, $L^\infty$-norm as well as for its gradient in $L^{2-a}$-norm for all $a\in (0,1)$. Optimal $L^2$-error estimates are shown for solutions under some additional regularity assumptions. Numerical results using the lowest order Raviart-Thomas mixed element confirm the theoretical analysis regarding convergence rates.     
\end{abstract}
            
 %\category{...}{...}{...}
            
%\terms{...} 
            
\begin{keywords}
Error estimates, expanded mixed finite element, nonlinear degenerate parabolic equations, generalized Forchheimer equations, porous media.
\end{keywords}

\begin{AMS}
65M12, 65M15, 65M60, 35Q35, 76S05.
\end{AMS}

\pagestyle{myheadings}
\thispagestyle{plain}
\markboth{Thinh T. Kieu}{Expanded mixed finite element method for the generalized Forchheimer flows}
            
%\begin{bottomstuff} 
%...
%\end{bottomstuff}
            
\myclearpage    
\section {Introduction }

Fluid flow in porous media is a great interest in many areas of
reservoir engineering, such as petroleum, environmental and groundwater
hydrology. Description of fluid flow behavior accurately in the porous media is
essential to the successful design and operation of projects in these areas. Most of  study of fluid flow in porous media are based on Darcy's law. By this law, the pressure gradient $\nabla p$ is linearly proportional to the fluid velocity $\uu$ in the porous media which writes as 
$\alpha \uu = -\nabla p$ with empirical constant $\alpha$. However  Dupuit, a Darcy's  student, observed on the field data that this linear relation is no
longer valid for flows owning high velocity. A nonlinear relationship between velocity and gradient of pressure is introduced by adding the higher order term of velocity to the Darcy's law. It is known as Forchheimer laws. Engineers widely use the three following Forchcheimmer's laws (cf. \cite{ ForchheimerBook}) to match experimental observation: 
\beqs
\alpha \uu +\beta |\uu| \uu= -\nabla p, \quad \alpha \uu +\beta |\uu| \uu+\lambda |\uu|^2 \uu= -\nabla p,\quad \alpha \uu +\lambda_m |\uu|^{m-1} \uu= -\nabla p,
\eeqs
where $\alpha, \beta, \lambda, m,\lambda_m$ are empirical constants.  

Since then, there is a large number of research on these equations and their variations, the  Brinkman-Forchheimer equations for incompressible fluids (cf. \cite{CKU2005,CKU2006,ChadamQin,Franchi2003,Straughan2013,Payne1999a,Payne1999b,Qin1998}, see also \cite{Straughanbook}). Recently, study on slightly compressible fluid flows subject to generalized Forchheimer equations are in \cite{ABHI1,HI1,HI2} and later in \cite{HIKS1,HK, HKP1}. These are devoted to theory of existence, stability and qualitative property of solutions. The study of numerical methods for degenerate parabolic equations are still not analyzed as much as those of theory.  

The popular numerical methods for modeling flow in porous media are the mixed finite element approximations in \cite {DW93, GW08,  KP99, EJP05} and block-centered finite difference method in \cite{HPH12} because these inherit conservation properties and produce the accurate flux (see \cite{ELPV96}). 

In \cite {ATWZ96}  Arbogast, Wheeler and Zhang first analyzed mixed finite element approximations of degenerate parabolic equation arising in flow in porous media. Not so long later Arbogast, Wheeler and Yotov in \cite {ATWY97} showed that the standard mixed finite element method not suitable for problems with small tensor coefficients as we need to invert the tensor. The proposed approach reduces original Forchheimer type equation to generalized Darcy equation with conductivity tensor $K$ degenerating as gradient of the pressure convergence to infinity. At the same time, the standard mixed variational formulation requires inverting $K$ to find gradient of pressure.    

Woodward and Dawson in \cite{ WCD00} study of expanded mixed finite element methods for a nonlinear parabolic equation modeling flow into variably saturated porous media. In their analysis, the Kirchhoff transformation is used to move the nonlinearity from coefficient $K$ to the gradient and thus simplifies analysis of the equations. This transformation does not applicable for our system \eqref{lin-p}.  

In this paper,  we combine techniques developed in \cite {HI1,HI2} and the expanded mixed finite element method as in \cite {ATWY97} to utilize both the special structures of equation as well as the advantages of the expanded mixed finite element method in obtaining the optimal order error estimates for the solution in several norms of interest. 

 The paper is organized as follows:  
  In \S 2 we introduce the generalized formulation of the Forchheimer’s laws for slightly compressible fluids, recall the relevant results from \cite{ABHI1, HI1} and preliminary results.   
 In \S 3 we consider the expanded mixed formulation and standard results for mixed finite element approximations. A implicit backward difference time discretization of the semidiscrete  scheme is proposed to solve the system \eqref{semidiscreteform}. 
 In \S 4 we derive many bounds for solutions to \eqref{weakform} and \eqref{semidiscreteform} in Lebesgue norms. 
 In \S 5 we analyze two version of a mixed finite element approximation, a semidiscrete version and a fully discrete version. The {\it priori} error estimates for the three relevant variables in $L^2$-norms, $L^\infty$-norm are established. Under suitable assumptions on the regularity of solutions, we prove the superconvergence. In \S 6, we provide a numerical example using the lowest Raviart-Thomas mixed finite element. The results support our theoretical analysis regarding convergence rates.  

\section{Mathematical preliminaries and auxiliaries}
%-----------------------------------------------------------------------------------------------
%Section 3
%-----------------------------------------------------------------------------------------------
We consider a fluid in a porous medium in a bounded domain $\Omega\subset \mathbb R^d, d\ge 2$. Its boundary $\Gamma=\partial \Omega$ belongs to $C^2$.  Let $x\in \mathbb{R}^d,$ $0<T<\infty,$  $t\in (0,T]$ be the spatial and time variable.   
     
A general Forchheimer equation, which is studied in \cite{ABHI1,HI1,HIKS1, HKP1} has the form 
 \beq\label{eq1}
g(|\uu|)\uu=-\nabla p,
\eeq
where $g(s)\geq 0$ is a function defined on $[0,\infty)$. When $g(s)=\alpha, \alpha+\beta s, \alpha+\beta s+\gamma s^2, \alpha +\gamma_m s^{m-1}, $ where $\alpha, \beta, \gamma, m,\gamma_m$ are empirical constants, we have Darcy's law, Forchheimer's two term, three term and power laws, respectively.
The function $g$ in \eqref{eq1} is a polynomial with non-negative coefficients as the form
\beq\label{eq2}
g(s)=a_0s^{\alpha_0} + a_1s^{\alpha_1}+\cdots +a_Ns^{\alpha_N},~~s\geq 0, 
\eeq 
where $N\geq 1,\alpha_0=0<\alpha_1<\ldots<\alpha_N$ are fixed number, the coefficients $a_0, \ldots, a_N$ are non-negative numbers with $a_0>0, a_N>0$. 
The number $\alpha_N$ is the degree of $g$ is denoted by ${\rm deg}(g)$. 

The monotonicity of the nonlinear term and the nondegeneracy of the Darcy's parts in \eqref{eq1} enable  us to write $\uu$ implicit in terms of $\nabla p$ and derivative of a nonlinear Darcy equation: 
\beq\label{eq3}
\uu=-K(|\nabla p|)\nabla p. 
\eeq 
The function $K: \mathbb{R}^+\rightarrow\mathbb{R}^+$ is defined by
\beq\label{eq4}
K(\xi)=\frac{1}{g(s(\xi))} \text{~where~} s=s(\xi)\geq 0 \text{~satisfies~} sg(s)=\xi, \text{~~for~}\xi\geq 0.
\eeq
%  When the dependence on $\vec{a}$ needs to be specified, we use notation $g(s,\vec{a}), K(\xi,\vec{a}, s(\xi,\vec{a}))$ to denote respectively functions in \eqref{eq2} and \eqref{eq3}. 
 The state equation, which relates the density $\rho(x,t)>0$ with pressure $p$, for slightly compressible fluids is 
   \beq\label{eq6}
   \frac{d\rho}{dp}=\kappa^{-1}\rho \text{~or~} \rho(p)=\rho_0\exp(\frac{p-p_0}{\kappa}),\quad \kappa>0.
   \eeq
Other equations govering the fluid's motion are the equation of continuity:
\beqs
\frac{d \rho}{d t} +\nabla\cdot(\rho \uu)=0,
\eeqs 
which yields 
\beq\label{eq5}
\frac {d\rho}{dp} \frac{d p}{d t} +\rho\nabla\cdot \uu + \frac {d\rho}{dp}\uu\cdot \nabla p=0. 
\eeq
Combining \eqref{eq5} and \eqref{eq6}, we find that
   \beq\label{eq7}
   \frac{d p}{d t}+\kappa \nabla\cdot \uu+\uu\cdot \nabla p=0.
   \eeq
   Since for most slightly compressible fluids in porous media the value of the constant $\kappa$ is large,
following engineering tradition we drop the last term in \eqref{eq7} and study the reduced equation,
\beq\label{lin-p} \frac{d p}{d t} +\kappa \nabla\cdot \uu=0.\eeq
By rescaling the time variable,  hereafter we assume that $\kappa=1$.  
    
Let  $\s=\nabla p$. Equations \eqref{lin-p} and \eqref{eq3} are equivalent to the system 
\beq\label{syseq}
\begin{cases}
     p_t +  \nabla \cdot \uu =0,\\
     \uu + K(|\s|)\s=0, \\
     \s- \nabla p = 0.
\end{cases}
\eeq 

The following properties of  function $K(\xi)$ are proved in Lemma III.5 and III.9 of \cite{ABHI1}, Lemma 2.1 and 5.2 of \cite{HI1} . 
\begin{lemma} We have for any $\xi\ge 0$ that 

(i) $K: [0,\infty)\to (0,a_0^{-1}]$ and it decreases in $\xi.$  

(ii) Type of degeneracy  
\beq\label{K1}  \frac{c_1}{(1+\xi)^a}\leq K(\xi)\leq \frac{c_2}{(1+\xi)^a}.  \eeq

(iii) For all $n\ge 1,$ 
\beq\label{K2} c_3(\xi^{n-a}-1)\leq K(\xi)\xi^n\leq c_2\xi^{n-a}.  \eeq

 (iv) Relation with its derivative 
           \beq\label{K3} -aK(\xi)\leq K'(\xi)\xi\leq 0. \eeq
  where $c_1, c_2, c_3$ are positive constants depending on $\Omega$ and $g$ and constant $$a=\frac{\alpha_N}{\alpha_N+1}=\frac{{\rm deg} (g)}{{\rm deg} (g)+1}\in(0,1).$$
\end{lemma}
We define 
\beq\label{Hdef}
H(\xi)=\int_0^{\xi^2} K(\sqrt{s}) dx, \text{~for~} \xi\geq 0. 
\eeq
The function $H(\xi)$ can compare with $\xi$ and $K(\xi)$ by
\beq
K(\xi)\xi^2 \leq H(\xi)\leq 2K(\xi)\xi^2, \label{H1}
\eeq
as a consequence of \eqref{K1}--\eqref{K2}
\beq
C(\xi^{2-a}-1) \leq H(\xi) \leq 2C\xi^{2-a}. \label{H2}
\eeq 
%==========================================%
Next we recall important monotonicity properties    
\begin{lemma}[cf. \cite{HI1}, Lemma 5.2]  For all $y, y' \in \mathbb{R}^d$,  one has 
\beq\label{Mono}
(K(|y'|)y' -K(|y|)y )\cdot(y'-y)\geq (1-a)K( \max\{ |y|, |y'|\} )|y' -y|^2 .
\eeq      
\end{lemma}
\begin{lemma}[cf. \cite{ABHI1}, Lemma III.11]
For the vector functions $\s_1, \s_2$, we have 
\beq\label{Monoprop}
\int_\Omega (K(|\s_1|)\s_1-K(|\s_2|)\s_2)\cdot(\s_1 -\s_2) dx\geq C\omega\norm{\s_1-\s_2}_{L^{2-a}(\Omega)}^2,
\eeq
where 
\beq
\omega =\left(1+ \max\{\|\s_1\|_{L^{2-a}(\Omega)} ,  \|\s_2\|_{L^{2-a}(\Omega)} \}\right)^{-a}.
\eeq
\end{lemma}
For the continuity of $K(\xi, \vec a)$ we have the following fact 
\begin{lemma}\label{conts} For all $y, y' \in \mathbb{R}^d$. There is a positive constant $C$ such that  
\beq\label{Lips}
   |K(|y'|)y' -K(|y|)y| \leq C|y' -y|.
\eeq  
\end{lemma}
\begin{proof}
{\it Case 1}: The origin does not belong to the segment connect  $y'$ and  $y$.  Let $\ell(t)=ty'+(1-t)y, t\in[0,1]$. Define $ h(t)=K(|\ell(t)|)\ell(t)$  for $t\in[0,1]$.
By the mean value theorem, there is $t_0\in[0,1]$ with $\ell(t_0)\neq 0$, such that
\begin{align*}
|K(|y'|)y' -K(|y|)y|^2 &= | h(1) - h(0)|^2  =|h'(t_0)|^2 \\
   &= \left| K'(|\ell(t_0)|)\frac{\ell(t_0)\cdot \ell'(t_0) }{|\ell(t_0)|}\ell(t_0)   +K(|\ell(t_0)|)\ell'(t_0) ) \right|^2.
\end{align*} 
Using \eqref{K3} and Minkowski's inequality we obtain    
\beqs   
|K(|y'|)y' -K(|y|)y|^2 \le 2 | K(|\ell(t_0)|)|^2\left\{a^2 \left|\frac{\ell(t_0)\cdot \ell'(t_0) }{|\ell(t_0)|^2}\ell(t_0)  \right|^2   + |\ell'(t_0) )|^2\right\}.  
\eeqs
The \eqref{Lips} follows by the boundedness of $K(\cdot)\le a_0^{-1}$.  

{\it Case 2}: The origin belongs to the segment connect $y', y$. We replace $y'$ by some $y_\varep\neq 0$ so that $0\notin [y_\varep, y]$ and $y_\varep \to 0$ as $\varep \to 0$. Apply the above inequality for $y$ and $y_\varep$, then let $\varep\to 0$.
\end{proof} 
%=======================================

{\bf Notations.} Let $L^2(\Omega)$ be the set of square integrable function on $\Omega$ and $( L^2(\Omega))^d$ the space of $d$-dimensional vectors which have all components in $L^2(\Omega)$. 

We denote by $(\cdot, \cdot)$ the inner product in either $L^2(\Omega)$ or $(L^2(\Omega))^d$ that is
\beqs
( \xi,\eta )=\int_\Omega \xi\eta dx \quad \text{ or } (\boldsymbol{\xi},\boldsymbol \eta )=\int_\Omega \boldsymbol{\xi}\cdot \boldsymbol{\eta} dx. 
\eeqs
and  $\langle \cdot, \cdot \rangle $ defined by    
\beqs
\langle u,v \rangle =\int_\Gamma uvd\sigma.  
\eeqs  
 The notation $\norm {\cdot}$ will means scalar norm $\norm{\cdot}_{L^2(\Omega)}$ or vector norm $\norm{\cdot}_{(L^2(\Omega))^d}$. 

For $1\le q\le +\infty$ and $m$ any nonnegative integer, let
\beqs
W^{m,q}(\Omega) = \{f\in L^q(\Omega), D^\alpha f\in L^q(\Omega), |\alpha|\le m \}
\eeqs 
denote a Sobolev space endowed with the norm 
\beqs
\norm{f}_{m,q} =
\left( \sum_{|\alpha|\le m} \norm{D^\alpha f}^q_{L^q(\Omega)} \right)^{\frac 1q}.
\eeqs    
Define $H^m(\Omega)= W^{m,2}(\Omega)$ with the norm $\norm{\cdot}_m =\norm{\cdot }_{m,2}$. 

For functions $p,u$ and vector functions $\uu,\s, \mathbf{v}$ we use short hand notations 
\beqs
\norm{p(t)} = \norm{ p(\cdot, t)}_{L^2(\Omega)},\quad \norm{\uu(t)} = \norm{ \uu(\cdot, t)}_{L^2(\Omega)}, \quad \norm{\s(t)}_{L^{2-a}} = \norm{ \s(\cdot, t)}_{L^{2-a}(\Omega)}  
\eeqs  
and 
\beqs
 u^0 =  u(\cdot,0), \quad  {\mathbf v}^0 =  {\mathbf v} (\cdot,0).  
\eeqs
for all functions $u$ and vector functions ${\mathbf v}$.   

Throughout this paper the constants 
$$\beta =2-a,\quad \lambda = \frac {2-\beta}\beta,\quad \delta= \frac{\beta}{\beta-1}.$$ 
The arguments $C, C_1$ will represent for positive generic constants and their values  depend on exponents, coefficients of polynomial  $g$,  the spatial dimension $d$ and domain $\Omega$, independent of the initial and boundary data, size of mesh and time step. These constants may be different place by place. %-----------------------------------------------------------------------------------------------
%Section 
%-----------------------------------------------------------------------------------------------
\myclearpage
\section{Expanded mixed finite element methods}

In this section, we develop the semidiscrete expanded mixed finite element method for the problem \eqref{lin-p} and a fully discrete version.  

%\subsection{Semidiscrete method}
Consider the initial value boundary problem (IVBP):  
\beq\label{eqorig}
\begin{cases}
     p_t +  \nabla \cdot \uu =f, \\
     \uu + K(|\s|)\s=0, \\
     \s- \nabla p = 0,
\end{cases}
\eeq 
for all $x\in  \Omega, t\in(0,T)$,  where $f: \Omega\times (0,T)\to \mathbb R$ is given and $f\in C^1([0,T];L^\infty(\Omega)).$ We assume the flux condition on the boundary:
$\uu\cdot \nu =0, x\in \Gamma,\ t\in[0,T],$
where $\nu$ is the outward normal vector on $\Gamma$.
% Hence, by \eqref{eq3} we have
%\beq\label{BC} -K(|\nabla p|)\nabla p \cdot \nu=0 \quad \hbox{on}\quad \lambda\times (0,\infty) .\eeq
The initial data: $p(x,0) = p_0(x)$ is given.

Let $W=L^2(\Omega)$, $\tilde W= (L^2(\Omega))^d,$ and the Hilbert space    
$$V= H_0({\rm div}, \Omega)= \left\{ \vv \in  (L^2(\Omega))^d, \nabla \cdot \vv \in L^2(\Omega), \vv\cdot  \nu =0 \text { on } \Gamma \right\}$$
with the norm defined by 
$
\norm{ \vv}_V^2 = \norm{\vv}^2 + \norm{ \nabla \cdot \vv }^2.
$

The variational formulation is defined as the following:  Find $(p,\s, \uu):[0,T] \rightarrow W\times \tilde W\times  V $ such that 
\begin{subequations}\label{weakform}
\begin{align}
&\label{W1}    ( p_t, w) +  \left(\nabla \cdot \uu, w\right) =\left(f,w\right),   && \forall w\in W,    \\
&\label{W2}     (\uu, \zz)  + ( K(|\s| )\s ,\zz )=0, && \forall \zz\in \tilde W,\\
 &\label{W3}     (\s, \vv)  + ( p , \nabla \cdot \vv )=  0 &&\forall  \vv\in V,
\end{align}
\end{subequations} 
with $p(x,0)=p_0(x),$ $x\in \Omega$ and  $ \uu\cdot \nu =0$, $ x\in \Gamma,\ t\in[0,T].$

Let $\{\mathcal T_h\}_h$ be a family of quasi-uniform triangulations of $\Omega$ with $h$ being the maximum diameter of the element.
Let $V_h$ be the Raviart-Thomas-N\'ed\'elec spaces ~\cite{Ned80,RavTho77a} of order $r\ge 0$ or Brezzi-Douglas-Marini spaces \cite{BDM85} of index $r$     over each triangulation $\mathcal T_h$, $W_h$ the space of discontinuous piecewise polynomials of degree $r$ over $\mathcal T_h$, $\tilde W_h$ the n-dimensional vector space of discontinuous piecewise polynomials of degree $r$ over $\mathcal T_h$.  Let $W_h \times  \tilde W_h\times V_h $ be the mixed element spaces approximating to $W\times\tilde W\times V $.     
The semidiscrete expanded mixed formulation of~\eqref{weakform} can read as following: Find $(p_h,\s_h,\uu_h):[0,T] \rightarrow W_h\times \tilde W_h\times V_h$ such that
\begin{subequations}\label{semidiscreteform}
\begin{align}
  &\label{Semi1}  (p_{h,t}, w_h) + \left(\nabla\cdot \uu_h,  w_h\right) =(f,w_h),   &&\forall w_h\in W_h,    \\
  &\label{Semi2}   (\uu_h,\zz_h)  + ( K(|\s_h| )\s_h   , \zz_h )=0,  &&\forall  \zz_h\in \tilde W_h,\\
  &\label{Semi3}    (\s_h,\vv_h)  + (  p_h , \nabla\cdot \vv_h )= 0,  &&\forall  \vv_h\in V_h,
\end{align}
\end{subequations} 
where $p_h(x,0)=\pi p_0(x).$    $ \uu_h\cdot \nu =0 $, $ x\in \Gamma,\ t\in[0,T].$
  
We use the standard $L^2$-projection operator 
$\pi: W \rightarrow W_h$,   $\pi: \tilde W \rightarrow \tilde W_h$ satisfying
\beq
( \pi w , \nabla \cdot \vv_h ) = ( w , \nabla \cdot \vv_h )  
\eeq
 for all  $w\in W, \vv_h \in V_h,$ and
\beq
   ( \pi \zz , \zz_h ) = ( \zz , \zz_h )
\eeq
 for all  $\zz\in \tilde W, \zz_h\in \tilde W_h.$ 
 
Also we use  $H$-div projection $\Pi : V
\rightarrow V_h$ defined by
\begin{equation}
( \nabla \cdot \Pi \vv , w_h )
= ( \nabla \cdot \vv , w_h )
\end{equation}
for all $w_h \in W_h$.

These projections have well-known approximation properties as in~\cite{BF91, JT81}. Below are the standard approximation properties for these projections  
 
(i) There exist positive constant $C_1, C_2$ such that
\begin{equation}
\label{prjpi}
\begin{split}
\norm{\pi w - w }_{0,\alpha} \leq C_1 h^m \norm{w}_{m,\alpha} \text{ and } 
\norm{ \pi \zz - \zz }_{0,\alpha} \leq C_2 h^m \norm{\zz}_{m,\alpha}, 
\end{split}
\end{equation}
for all $w \in W^{m,\alpha}(\Omega)$, $\zz\in (W^{m,\alpha}(\Omega))^d$,  $0\le m \le r+1, 1\le\alpha\le \infty$. Here $\norm {\cdot}_{m,\alpha}$ denotes a standard norm in Sobolev space $W^{m,\alpha}$. In short hand, when $\alpha=2$ we write \eqref{prjpi} as   
\begin{equation}
%\label{prjpi}
\begin{split}
\norm{\pi w - w } \leq C_1 h^m \norm{w}_{m}, \quad  \text { and }\quad \norm{ \pi \zz - \zz } \leq C_2 h^m \norm{\zz}_m. 
\end{split}
\end{equation}

(ii) There exists a positive $C_3$ such that
\begin{equation}
\label{prjPi}
\norm{\Pi \vv - \vv}_{0,\alpha}  \leq C_3 h^m \norm{ \vv }_{m,\alpha}
\end{equation}
for any $\vv \in \left( W^{m,\alpha} ( \Omega ) \right)^d,$  $1/\alpha \leq m \leq r+1$, $ 1\le \alpha \le \infty $.

Because of the commuting relation between $\pi, \Pi$ and the
divergence (i.e., that $\nabla \cdot \Pi \uu = \pi ( \nabla \cdot \uu
)$, we also have the bound
\begin{equation}
\label{prjdiv}
\| \nabla \cdot ( \Pi \vv - \vv ) \|_{0,\alpha} \leq C_1 h^m \norm{\nabla \cdot \vv }_{m,\alpha},
\end{equation}
provided $\nabla \cdot \vv \in W^{m,\alpha}(\Omega)$ for $1 \leq m \leq r+1$.
%\subsection{Fully discrete method.}

Let $N$ be the positive integer, $t_0=0 < t_1 <\ldots < t_n= T$ be partition interval $[0,T]$ of $N$ sub-intervals, and let $\Delta t = t_{n} - t_{n-1}=T/N$  be the $n$-th time step size, $t_n=n\Delta t$ and $\varphi^n = \varphi(\cdot, t_n)$. 

The discrete time expanded mixed finite element approximation to \eqref{weakform} is defined as follows:  Find $ (p_h^n,\s_h^n, \uu_h^n)\in W_h\times \tilde W_h\times V_h$, $n=1,2,\dots, N$, such that 
\begin{subequations}\label{fullydiscreteform}
\begin{align}
&\label{fully1}    \left( \frac{ p_h^n -  p_h^{n-1}}{\Delta t }, w_h\right) + \left(\nabla \cdot \uu_h^n, w_h\right) =(f^n , w_h ),   &&\forall w_h\in W_h, \\
&\label{fully2}     (\uu_h^n, \zz_h)  + ( K(|\s_h^n| )\s_h^n ,\zz_h )=0, &&\forall \zz_h\in {\tilde W}_h,\\
 &\label{fully3}     (\s_h^n,{\bf v_h})  + ( p_h^n , \nabla\cdot \vv_h )=0, &&\forall \vv_h\in V_h.
\end{align}
\end{subequations} 
The initial approximations are chosen by
$$
p_h^0(x)=\pi p_0(x) ,  \quad \s_h^0(x)=\nabla p_h^0(x) ,\quad  \uu_h^0(x) =K(|\s_h^0(x)|)\s_h^0(x) 
$$
for all $x\in \Omega.$
%
%%------------------------------------------------------------
\section{Estimates  of solutions }
 Using the theory of monotone operators \cite{MR0259693,s97,z90}, the authors in \cite {HIKS1} proved the global existence of weak solution of equation \eqref{eqorig}. Moreover $p\in C([0,T),L^\alpha(\Omega))$, $\alpha\ge 1$ and $L_{loc}^{\beta}([0,T),W^{1,\beta}(\Omega))$  and  $p_t \in L^{\beta'}_{loc}\bigl([0,T), (W^{1,\beta}(\Omega))'\bigr)\cap  L^2_{loc}\bigl([0,T), L^2(\Omega)\bigr)$
  provided the initial, boundary data and $f$  sufficiently smooth. For  a { \it priori } estimate, we assume that the weak solution is sufficiently regularities both in $x$ and $t$ variables. 
%%%%%%%%%%%%%%%%%%%%%%%%%%%%%%%%%%%%%%%%%%%%%%

%----------------------------------------------------------
\begin{theorem}\label{ph} Let $(p,\s,\uu)$ be the solution to the problem \eqref{weakform}.  We have 

{\rm (i)}  
\beq\label{B4p}
\sup_{t\in[0,T]}\norm{p(t)}  \le  \norm {p^0} +\int_0^T \norm {f(t)} dt. 
\eeq 

{\rm (ii)} For any $t\in (0,T)$,
\beq\label{Bpt}
\int_0^t \norm{p_t(t) }^2 dt +  \int_\Omega H(x,t)dx + \norm{p(t)}^2 \le C\mathcal M(t), 
\eeq
where
\beq\label{BM}
\mathcal M(t) = \norm{\s^0}^{\beta}_{L^{\beta}(\Omega)} + \norm{p^0}^2+ 2 \int_0^t\norm{f(t)}^2+ t\left(\norm {p^0} +\int_0^T \norm {f(t)} dt\right).
\eeq
{\rm (iii)} For any $t\in (0,T)$, 
\beq\label{Bs}
 \begin{aligned}
 \norm{\s(t)}^{\beta}_{L^{\beta}(\Omega)} +\norm{\uu(t)} \le C \Big\{\norm{p^0}^2 + \Big(\int_0^T \norm{f(t)} dt\Big)^2 &\\
   + \int_0^t e^{-(t-\tau)}\norm{f(\tau)}^2 d\tau +1\Big\}&.
\end{aligned}
 \eeq
\end{theorem}
%==========================
\begin{proof}
(i) In \eqref{semidiscreteform}, picking up $w= p$, $\zz=\s$ and $\vv=\uu$ we have 
\begin{subequations}
\begin{align}
  \label{eq4p}  &( p_t,  p) + \left( \nabla\cdot \uu, p\right) =(f, p), \\
  \label{eq4u}&(\uu,\s)  + ( K(|\s| )\s   , \s )=0,  \\
  \label{eq4s}&(\s,\uu)  + (  p ,\nabla\cdot \uu )= 0.
\end{align}
\end{subequations}
We add three above equations to obtain
\beq\label{maineq}
\frac 12 \frac{d}{dt} \norm{p}^2 + \norm{ K^{\frac 12}(|\s| )\s}^2 = (f, p).
\eeq
For each $t\in [0,T)$, integrating the previous estimate on $(0,t)$ and taking the supremum in $t$ yield      
 \beq\label{dervp}
 \begin{aligned}
 \sup_{t\in[0,T]}\norm{p(t)}^2  + 2\int_0^T \norm{ K^{\frac 12}(|\s| )\s}^2 &\le \norm {p(0)}^2 +\int_0^T (f,p) dt\\
 & \le \norm {p(0)}^2 + \sup_{t\in[0,T]}\norm {p(t)}\int_0^T \norm {f} dt. 
\end{aligned}
\eeq  
Dropping the nonnegative term of the left-hand side of \eqref{dervp}, we have the bound 
 \beqs
\sup_{t\in[0,T]}\norm{p(t)}^2 \le \norm {p(0)}^2 + \sup_{t\in[0,T]}\norm {p(t)}\int_0^T \norm {f} dt. 
\eeqs   
This have the form $x^2 \le \delta^2 +\eta x$ where 
\begin{align*}
x&= \sup_{t\in[0,T]}\norm{p(t)}^2,\\
\eta &=\int_0^T \norm {f} dt \ge 0, \\
\delta &= \norm {p(0)}\ge 0. 
\end{align*} 
The element quadratic inequality shows that  $x\le \delta+\eta $. Hence it proves \eqref{B4p}.  

(ii) Selecting $w=p_t$, $\zz=\s_t$ in \eqref{Semi1}, \eqref{Semi2},   differentiating \eqref{Semi3} in time and then choosing $\vv_h=\uu$, we obtain    
\begin{align*}
  &( p_t,  p_t) + \left(\nabla\cdot \uu,  p_t\right) =(f, p_t), \\
  &(\uu,\s_t)  + ( K(|\s| )\s   , \s_t )=0, \\
  &(\s_t,\uu)  + (p_t ,  \nabla \cdot \uu )= 0.
\end{align*}
Summing up three equations gives  
\beq\label{ptEq}
\norm{p_t }^2 + ( K(|\s| )\s   , \s_t ) = (f, p_t).
\eeq 
Note that the function $H(\cdot)$ in \eqref{Hdef} gives      
$$ K(|\s| )\s  \cdot \s_t =\frac12\frac d{dt} H(\s).$$
We rewite \eqref{ptEq} as  
\beq\label{ptest}
\norm{p_t }^2 +  \frac 12 \frac d {dt}\int_\Omega H(x,t)dx = (f, p_t),  
\eeq
where $H(x,t)=H(\s(x,t)).$  

Now adding \eqref{maineq} and \eqref{ptest} we obtain  
\beq \label{neweq}
\norm{p_t }^2 +\frac 12 \frac d {dt}\left(\int_\Omega H(x,t)dx + \norm{p}^2\right)+ \norm{ K^{\frac 12}(|\s| )\s}^2  = (f,p)+(f,p_t). 
\eeq
Using Cauchy's inequality and the fact that % \eqref{H1} to compare $K(|\s| )\s^2$ to $H(s)$, one gets 
\beqs
\norm{ K^{\frac 12}(|\s| )\s}^2 =\int_\Omega K(|\s| )\s^2 dx  \ge \frac 12\int_\Omega H(\s(x,t)) dx. 
\eeqs
It follows  that     
\beq\label{estpt}
\norm{p_t }^2 +  \frac d {dt}\left(\int_\Omega H(x,t)dx + \norm{p}^2\right)\le - \int_{\Omega}H(x,t)dx+  2\norm{f}^2+\norm {p}^2. 
\eeq
Integrating above inequality in $t$, using \eqref{B4p}, we find that   
\begin{multline}\label{estp-t}
\int_0^t \norm{p_t(\tau) }^2 d\tau +  \int_\Omega H(x,t)dx + \norm{p}^2\le -\int_0^t\int_{\Omega}H(x,t)dx\\
+ \int_\Omega H(x,0)dx + \norm{p(0)}^2+ 2 \int_0^t\norm{f}^2+ t\left(\norm {p(0)} +\int_0^T \norm {f} dt\right) . 
\end{multline}
Dropping the negative term on the right hand side of \eqref{estpt} and using the fact that $H(x,0)\le C|\s(x,0)|^{\beta}$ we obtain \eqref{Bpt}. 

(iii) We rewrite equation \eqref{neweq} as form
\beqs 
\begin{split}
\norm{p_t }^2 +\frac 12 \frac d {dt}\int_\Omega H(x,t)dx &=- \norm{ K^{\frac 12}(|\s| )\s}^2 + (f,p+p_t) -(p,p_t)\\
&\le -\frac 12 \int_\Omega H(x,t) + \frac 12 \left(\norm{f}^2 +\norm{p}^2 +\norm{p_t}^2 \right).
\end{split} 
\eeqs
This implies    
\beqs
\frac d {dt}\int_\Omega H(x,t)dx \le -\int_\Omega H(x,t) + \norm{f}^2 +\norm{p}^2. 
\eeqs
Applying Gronwall's inequality, we obtain 
\beqs
\int_\Omega H(x,t)dx \le  -e^{-t}\int_\Omega H(x,0)dx+C\int_0^t e^{-(t-\tau)}( \norm{f}^2+\norm{p}^2) d\tau .
\eeqs
Dropping the first term of the right hand side, using \eqref{B4p}, we  obtain 
\beq\label{res1}
\begin{split}
\int_\Omega H(x,t)dx &\le  C\int_0^t e^{-(t-\tau)}\left\{ \norm{f}^2+ \norm{p(0)}^2 + \left(\int_0^T \norm{f(s)} ds\right)^2 \right\} d\tau\\
&\le C \left\{\norm{p(0)}^2 + \left(\int_0^T \norm{f(t)} dt\right)^2  + \int_0^t e^{-(t-\tau)}\norm{f(\tau)}^2 d\tau \right\}.
\end{split}
\eeq
Note that 
\beq\label{B4s}
\int_\Omega H(x,t)dx \ge C\int_\Omega (|\s|^{\beta} -1) dx= C(\norm{\s}^{\beta}_{L^{\beta}(\Omega)}-1).
\eeq
In addition equation \eqref{eq4u} leads to  
\beq\label{B4u}
\norm {\uu} \le \norm {K^{\frac 12}(|\s|)\s}\le C \norm{\s}^{\beta}_{L^{\beta}(\Omega)}.
\eeq
Therefore, \eqref{Bs} follows from \eqref{res1}, \eqref{B4s} and \eqref{B4u}. The proof is complete. 
\end{proof}
%----------------------------------------

Although solution is considered continuous at $t=0$ in appropriate Lebesgue or Sobolev space. Its time derivative is not. In the following       
we prove the time derivative solution is bounded. 
\begin{theorem}\label{phderv} Let $0<t_0<T$. For each $t\in [t_0, T],$ we have 
\beq\label{Bp-t}
\norm{  p_t(t)}^2\le  C t_0^{-1}\mathcal M(t_0)+ C \left(\mathcal M(t)+\int_0^t  (\norm{f_t(\tau)}^2 d\tau\right),
\eeq
where $\mathcal M(\cdot)$ is defined as in \eqref{BM}. 
\end{theorem}
\begin{proof}
We differentiate \eqref{weakform} with respect time $t$ to obtain  
\begin{subequations}
\begin{align}
&\label{ds1}    (  p_{tt}, w) +  \left(\nabla\cdot\uu_t, w\right) =(f_t, w) , &&\forall w\in W, \\
&\label{ds2}     (\uu_t, \zz)  +  ( K(|\s|)\s_t, \zz) +\left(K'(|\s|) \frac{\s\cdot \s_t}{|\s|}\s,\zz \right)=0, &&\forall \zz\in \tilde W, \\
 &\label{ds3}     (\s_t,\vv_h)  + (  p_t , \nabla\cdot  \vv )= 0, &&\forall \vv\in V.
\end{align}
\end{subequations}
For each $t\in[t_0,T]$, taking $w= p_t$, $\zz=\s_t$ and $\vv =\uu_t$, summing three resultant equations we obtain 
\beq\label{I}
\begin{split}
\frac12\ddt  \norm{  p_t}^2+ \norm{K^{\frac 12}(|\s|)\s_t}^2 &= - \left(K'(|\s|) \frac{\s\cdot \s_t}{|\s|}\s,\s_t \right)+(f_t, p_t).
\end{split}
\eeq
Using \eqref{K3} and Cauchy's inequality to bound the right hand-side of \eqref{I} give
\beq\label{I1}
\left| -\left(K'(|\s|) \frac{\s\cdot \s_t}{|\s|}\s,\s_t \right)+(f_t, p_t)\right|\le  a\norm{K^{\frac 12}(|\s|) \s_t}^2+ \frac 12\left(\norm{f_t}^2 +\norm{ p_t}^2 \right).
\eeq
Thus
\beqs
\begin{split}
\frac 12 \ddt \norm{ p_t}^2+ (1-a) \norm{K^{\frac 12}(|\s|)\s_t}^2 \le\frac 12 \left(\norm{f_t}^2 +\norm{ p_t}^2\right). 
\end{split}
\eeqs
 Ignoring the the nonnegative term of the left hand side in previous inequality we find that  
 \beq\label{DEpht}
\frac d{dt}  \norm{  p_t}^2 \le   \norm{ p_t}^2+\norm{f_t}^2.
\eeq
For $t\ge t'>0$,  integrating \eqref{DEpht} from $t'$ to $t$ yields 
\begin{align*}
\norm{p_t}^2&\le \norm{  p_t(t')}^2 + \int_{t'}^t \norm{ p_t}^2 d\tau +\int_0^t \norm{f_t}^2 d\tau\\
 &\le \norm{  p_t(t')}^2 + \int_{0}^t \norm{ p_t}^2 d\tau +\int_0^t \norm{f_t}^2 d\tau .
\end{align*}
Now integrating in $t'$ from $0$ to $t_0$,  
\beq\label{ptt0}
\begin{split}
t_0\norm{  p_t}^2&\le \int_0^{t_0} \norm{  p_t(t')}^2 + t_0\Big\{\int_{0}^t \norm{ p_t}^2 d\tau+ \int_0^t\norm{f_t}^2 d\tau  \Big\}.
\end{split}
\eeq
Combining \eqref{ptt0} and \eqref{Bpt} leads to \eqref{Bpht}.
The proof is complete.  
\end{proof}
%%---------------------------------------------------------------

Using $L^2$-projection, $H$-div projection and above arguments, the similar results for solution of discrete problem are established as following.    

\begin{theorem}\label{pspt} Let $(p_h,\s_h,\uu_h)$ be the solution to the semidiscrete  problem \eqref{semidiscreteform}.  We have

{\rm (i)}  
\beq\label{B4ph}
\sup_{t\in[0,T]}\norm{p_h(t)}^2  \le  \norm {p^0} +\int_0^T \norm {f(t)} dt. 
\eeq 

{\rm (ii)} For any $t\in (0,T)$, 
\beq\label{Bsh}
 \begin{aligned}
 \norm{\s_h(t)}^{\beta}_{L^{\beta}(\Omega)} +\norm{\uu_h(t)} \le C \Big\{\norm{p^0}^2 + \Big(\int_0^T \norm{f(t)} dt\Big)^2 &\\
    + \int_0^t e^{-(t-\tau)}\norm{f(\tau)}^2 d\tau +1\Big\}&.
\end{aligned}
 \eeq 
 
{\rm (iii)} Let $0<t_0<T$.  For any $t\in [t_0, T],$ we have 
\beq\label{Bpht}
\norm{  p_{h,t}(t)}^2\le  C t_0^{-1}\mathcal M(t_0)+ C \left(\mathcal M(t)+\int_0^t  (\norm{f_t(\tau)}^2 d\tau\right),
\eeq
where $\mathcal M(\cdot)$ is defined as in \eqref{BM}. 
\end{theorem}
%--------------------------------------------------
 \section{Error analysis}
     
In this section, we will establish the error estimates between the analytical solution and approximation solution in several norms. 
%We combine the techniques used in \cite{HI2} with the expanded mixed finite element method. 
In the below development we discuss error estimates for the case conductivity tensor $K(\cdot)$ degenerating. We assume the solutions, 
\beqs
p\in L^\infty(0,T; H^{r+1}(\Omega)) ,\quad \s \in L^2(0,T; (W^{r+1,\beta}(\Omega))^d).
\eeqs         

%  while the regularity properties of solution of \eqref{maineq} are not completely understood. In fact, the results in \cite{HKP1, HKP2} showed that $ p(x; t)\in L^\infty((0,T);L^\infty(\Omega))\cap L^\infty((0, T);W^{1,2-a}(\Omega))$ and $p_t(x; t)\in L^\infty_{loc}((0, T);L^2(\Omega))$. However, this is not enough for our analysis error estimates of the solutions. Therefore, we make at least the following assumption on regularity of solutions: 
%\begin{align}
% p\in L^2((0,T); H^1(\Omega) ), \quad p_t\in L^2((0,T); H^{-1}(\Omega) ),\\
% \uu\in L^\alpha((0,T); (H^1(\Omega))^d), \quad \s\in L^\alpha((0,T); (H^1(\Omega))^d)????????.
%\end{align}

 \subsection{Error estimate for semidiscrete method} 
We find the error bounds in the semidiscrete method by
comparing the computed solution to the projections of the true
solutions.  To do this, we restrict the test functions
in~\eqref{weakform} to the finite dimensional spaces.   
Let 
\begin{align*}
 p_h - p =(p_h-\pi p ) + (\pi p - p) \equiv  \vartheta + \theta,\\
  \s_h -\s = (\s_h-\pi \s ) + (\pi \s - \s) \equiv  \eta+ \zeta,\\
  \uu_h -\uu = (\uu_h-\Pi \uu ) + (\Pi \uu - \uu) \equiv  \rho+ \varrho.
\end{align*}
Properties of projections in \eqref{prjpi} and \eqref{prjPi} yield 
\begin{align}
\label{Btheta}
&\norm{\theta }_{L^\alpha(\Omega)}  \le Ch^m \norm{ p }_{m,\alpha},&&\forall p\in W^{m,\alpha}(\Omega),\\
\label{Bzeta}
&\norm{\zeta }_{L^\alpha(\Omega)}  \le Ch^m \norm{\s }_{m,\alpha},  &&\forall \s\in  (W^{m,\alpha}(\Omega))^d,\\
\label{Bvarrho}
&\norm{\varrho }_{L^\alpha(\Omega)}  \le Ch^m \norm{\uu }_{m,\alpha}, &&\forall \uu\in (W^{m,\alpha}(\Omega))^d. 
 \end{align}
for all $1\le m\le r+1$, $1\le \alpha\le \infty.$
%============================

Let $0<t_0<T$,
\begin{align*}
\mathcal A&=1+\norm{p^0}^2 + \left(\int_0^T \norm{f(t)} dt\right)^2 + \int_0^T \norm{f(t)}^2 dt.\\ 
\mathcal B &=C t_0^{-1}\mathcal M(t_0)+ C \left( \mathcal M(T)+\int_0^T \norm{f_t(t)}^2 dt\right). 
\end{align*}
%==============================
\begin{theorem}\label{mainres} Assume $(p^0,\uu^0,\s^0 )\in W\times V\times \tilde W$ and $(p_h^0,\uu^0_h,\s^0_h )\in W_h\times V_h\times \tilde W_h$. Let $(p, \uu,\s)$ solve problem \eqref{weakform} and $(p_h, \uu_h,\s_h)$ solve the semidiscrete mixed finite element approximation \eqref{semidiscreteform}.  Then there is a positive constant $C$ such that for each $t\in (0,T)$
\beq \label{mi4a}
\norm {(p_h - p)(t)  }\le Ch^{r+1}\norm{p(t)} + C   \mathcal A^{\frac12} h^{\frac {r+1}2}\sqrt{\int_0^t\norm{\s(\tau)}_{L^\beta(\Omega)}d\tau} . 
\eeq
Furthermore if $\s \in L^2(0,T; (W^{r+1,\delta}(\Omega))^d)$ then 
\beq\label{impr} 
\norm {(p_h - p)(t)  }\le Ch^{r+1} \norm{p(t)} + C \mathcal A^{\frac \lambda2} h^{r+1}\sqrt{\int_0^t \norm{\s(\tau)}_{L^{\delta}(\Omega) }^2 d\tau}. 
\eeq
\end{theorem}
%=======================
\begin{proof} 
Subtracting \eqref{semidiscreteform} from \eqref{weakform} we have the following error equations  
\begin{subequations}\label{Erreq}
\begin{align}
&\label{Erreq1} ( p_{h,t} - p_t , w_h) +  \left(\nabla \cdot(\uu_h -\uu) ,  w_h\right) =0,    &&\forall w_h\in W_h,\\
&\label{Erreq2} (\uu_h - \uu,\zz_h)  + \left( K(|\s_h| )\s_h - K(|\s|) \s, \zz_h \right)=0,  &&\forall \zz_h\in \tilde W_h,\\
&\label{Erreq3}  (\s_h-\s,\vv_h)  + (  p_h - p , \nabla \cdot \vv_h )= 0, 
&&\forall \vv_h\in V_h.
\end{align}
\end{subequations} 
%where initial condition is defined by $\bar{p}_h(x,0)=\pi \bar{p}_0(x)$  in $\Omega$.

Let take $w_h=\vartheta,$ $\zz_h=\eta $ and $\vv_h=\rho$. Using the projections in \eqref{prjpi} and \eqref{prjPi}, we rewrite \eqref{Erreq} as  
\begin{subequations}\label{eqprj}
\begin{align}
\label{prj1} &( \vartheta_t, \vartheta) +  \left(\nabla \cdot\rho, \vartheta \right) =0, \\
\label{prj2}&(\rho, \eta)  + \left( K(|\s| )\s -K(|\s|)\s  ,  \eta\right)=0,\\ 
\label{prj3} &(\eta,\rho)  + (\vartheta , \nabla\cdot  \rho )= 0.
\end{align}
\end{subequations} 
%where $\vartheta(0)=0.$ 
Summing up three equations \eqref{prj1}--\eqref{prj3} gives  
\beqs
\frac 12  \frac d{dt}\norm {\vartheta}^2 + \left( K(|\s_h| )\s_h -K(|\s|) \s ,  \eta \right) =0.
\eeqs
It is equivalent to
\beq\label{eqerr}
\frac 12  \frac d{dt}\norm {\vartheta}^2 + \left( K(|\s_h| )\s_h - K(|\s|) \s ,  \s_h - \s \right) =\left( K(|\s_h| )\s_h - K(|\s|) \s , \zeta \right).
\eeq
Applying \eqref{Monoprop} to the second term of \eqref{eqerr} we have    
\beq\label{Mn1}
\left( K(|\s_h| )\s_h - K(|\s| \s), \s_h -\s \right)\ge C\omega \norm{\s_h-\s}^2_{L^\beta(\Omega)}
\eeq
with
\beq\label{omega}
\omega = \omega(t) = (1+ \max\{\norm{\s_h(t)}_{L^\beta(\Omega)} ,  \norm{\s(t)}_{L^\beta(\Omega) } \}  )^{-a}.
\eeq

Since $K(|\xi| )\xi \le C\xi^{\beta-1} $, the right hand side of \eqref{eqerr} is bounded by 
\beqs
\begin{aligned}
\left| ( K(|\s_h| )\s_h - K(|\s|) \s, \zeta) \right|&\le C\left(|\s_h|^{\beta-1} +|\s|^{\beta-1},|\zeta|\right)\\
&\le C \left((\norm{\s_h}_{L^\beta(\Omega)}^{\beta-1}+\norm{\s}_{L^\beta(\Omega)}^{\beta-1} \right)  \norm{\zeta}_{L^\beta(\Omega)}\\
&\le C\left(1+ \norm{\s_h}_{L^\beta(\Omega)}^{\beta} +  \norm{\s}_{L^\beta(\Omega)}^{\beta}  \right)\norm{\zeta}_{L^{\beta}(\Omega)}.
\end{aligned}
\eeqs
Due to \eqref{Bs} and \eqref{Bsh},   
\beq\label{invome}
\begin{split}
1+ \norm{\s_h}_{L^\beta(\Omega)}^{\beta} +  \norm{\s}_{L^\beta(\Omega)}^{\beta}&\le C\left[1+\norm{p^0}^2 + \left(\int_0^T \norm{f} dt\right)^2+ \int_0^t e^{-(t-\tau)}\norm{f}^2 d\tau \right] \\
 &\le C\mathcal A.
\end{split}
\eeq
Hence
\beq\label{mi1a}
\left| ( K(|\s_h| )\s_h - K(|\s|) \s, \zeta) \right| \le C\mathcal A\norm{\zeta}_{L^{\beta}(\Omega)} 
\eeq
Combining \eqref{eqerr}, \eqref{Mn1} and \eqref{mi1a} leads to  
\beq\label{keya}
\begin{split}
\frac d{dt}\norm {\vartheta}^2 +  \omega\norm{\s_h -\s}^2_{L^\beta(\Omega)} 
\le  C\mathcal A\norm{\zeta}_{L^{\beta}(\Omega)} .
\end{split}
\eeq
Integrating \eqref{keya} in time, using $\vartheta(0)=0$,  we have
\beq\label{keyb}
\norm {\vartheta}^2 +\int_0^t  \omega\norm{\s_h -\s}^2_{L^\beta(\Omega)} d\tau\le C \mathcal A\int_0^t  \norm{\zeta}_{L^\beta(\Omega)}d\tau .
\eeq 
Ignoring the second term of \eqref{keyb} and using the triangle inequality     
$\norm{p_h -  p}\le \norm{\vartheta}+\norm{\theta}$
 we obtain
 \beq\label{mi2a}
\norm {p_h - p}^2 \le \norm{\theta}^2 + C\mathcal A \int_0^t \norm{\zeta}_{L^\beta(\Omega)}d\tau,
\eeq
which proves \eqref{mi4a}.

 Under the assumption more on the regularity of solution we bound the right hand side of \eqref{eqerr} using \eqref{Lips}, H\"older and Young's inequality to obtain   
\beq\label{mi1aa}
\begin{aligned}
\left| ( K(|\s_h| )\s_h - K(|\s|) \s, \zeta) \right|&\le C(|\s_h-\s|,|\zeta|)\\
&\le C \norm{\s_h-\s}_{L^\beta(\Omega)}\norm{\zeta}_{L^{\delta}(\Omega) }  \\
&\le \varep \norm{\s_h-\s}_{L^\beta(\Omega)}^2 +C\varep^{-1} \norm{\zeta}_{L^{\delta}(\Omega) }^2
\end{aligned}
\eeq
for all $\varep>0$. 

From \eqref{eqerr}, \eqref{Mn1} and \eqref{mi1aa}, we find that 
\beqs
\begin{split}
\frac d{dt}\norm {\vartheta}^2 +  C\omega\norm{\s_h -\s}^2_{L^\beta(\Omega)} 
\le \varep \norm{\s_h-\s}_{L^\beta(\Omega)}^2 +C\varep^{-1} \norm{\zeta}_{L^{\delta} (\Omega)}^2.
\end{split}
\eeqs
Due to \eqref{omega} and \eqref{invome},   
\beq\label{invomega}
\omega^{-1}\le C_1\left(1+ \norm{\s}_{L^\beta(\Omega)}^{\beta} +  \norm{\s_h}_{L^\beta(\Omega)}^{\beta}  \right)^\lambda\le C_1\mathcal A^\lambda.
\eeq
Thus 
\beq\label{keyaa}
\begin{split}
\frac d{dt}\norm {\vartheta}^2 +  C(C_1\mathcal A^\lambda)^{-1}\norm{\s_h -\s}^2_{L^\beta(\Omega)} 
\le\varep \norm{\s_h-\s}_{L^\beta(\Omega)}^2 +C\varep^{-1} \norm{\zeta}_{L^{\delta}(\Omega) }^2.
\end{split}
\eeq
Selecting $\varep = \frac C{2C_1\mathcal A^\lambda}$, integrating \eqref{keyaa} in time, we have
\beq\label{keyb1}
\norm {\vartheta}^2 +(C\mathcal A^\lambda)^{-1}\int_0^t \norm{\s_h -\s}^2_{L^\beta(\Omega)} d\tau\le C\mathcal A^\lambda \int_0^t  \norm{\zeta}_{L^{\delta}(\Omega) }^2 d\tau .
\eeq 
Dropping the second term in \eqref{keyb1} and using triangle inequality $ \norm {p_h - p} \le \norm{\theta}+\norm{\vartheta} $ in \eqref{keyb1} shows that   
 \beqs%\label{mi2aa}
\norm {p_h - p}^2 \le C\left( \norm{\theta}^2 + \mathcal A^\lambda \int_0^t \norm{\zeta}_{L^{\delta}(\Omega)}^2d\tau\right).
\eeqs
This together \eqref{Bzeta} and \eqref{Bvarrho} gives \eqref{impr}.
 The proof is complete.    
\end{proof}

%----------------------------------------------------------
The $L^2$-error estimate and the inverse estimate enable us to have the $L^\infty$- error estimate as the following 
%-----------------------------------------------------------
\begin{theorem}\label{pinferr}  Assume $(p^0,\uu^0,\s^0 )\in W\times V\times \tilde W$ and $(p_h^0,\uu^0_h,\s^0_h )\in W_h\times V_h\times \tilde W_h$. Let $(p, \uu,\s)$ solve problem \eqref{weakform} and $(p_h, \uu_h,\s_h)$ solve the semidiscrete mixed finite element approximation \eqref{semidiscreteform}. If $p\in L^\infty(0,T, W^{r+1,\infty}(\Omega))$ then there exists a positive constant $C$ such that for each $t\in(0,T)$, 
\beq\label{errlinf}
\norm{(p-p_h)(t)}_{L^\infty(\Omega)} \le Ch^{r+1}\norm{p(t)}_{r+1, \infty} + C\mathcal A^{\frac12} h^{\frac {r-1}2}\sqrt{\int_0^t  \norm{\s(t)}_{r+1,\beta}}.
\eeq
Furthermore  if $\s \in L^2(0,T; (W^{r+1,\delta}(\Omega))^d)$ then 
\beq\label{errlinf1}
\norm{(p-p_h)(t)}_{L^\infty(\Omega)} \le Ch^{r+1}\norm{p(t)}_{r+1, \infty} + C\mathcal A^{\frac12} h^{r}\sqrt{\int_0^t  \norm{\s(t)}_{r+1,\delta}^2}.
\eeq
\end{theorem}
%----------------------------------------------
\begin{proof}  
%According to \eqref{Btheta},
%\beq\label{term1}
%\norm{\theta }_{L^\infty(\Omega)} \le Ch^{r+1}\norm{p}_{r+1, \infty}.
%\eeq
For quasi-uniformly of $\mathcal T_h$, the following inverse estimate holds 
\beqs
\norm{\vartheta}_{L^\infty(\Omega)} \le Ch^{-\frac 2 q}\norm{ \vartheta}_{L^q(\Omega)} \quad \text { for all } 1\le q\le\infty.  
\eeqs
 Applying this with $q=2$ and using \eqref{keyb} imply 
 \beq\label{eb}
 \begin{split}
 \norm{\vartheta}_{L^\infty(\Omega)} \le Ch^{-1}\norm{ \vartheta}\le C\mathcal A^{\frac12} h^{-1} \left(\int_0^t  \norm{\zeta}_{L^\beta(\Omega)}\right)^{\frac 12}.
 \end{split}
 \eeq 
 It follows from triangle inequality and \eqref{eb} that
 \beq\label{varthinf}
\begin{split}
\norm{p-p_h}_{L^\infty(\Omega)} &\le \norm{\theta }_{L^\infty(\Omega)} +\norm{\vartheta}_{L^\infty(\Omega)}\\
 &\le \norm{\theta}_{L^\infty(\Omega)} +C\mathcal A^{\frac12} h^{-1} \left(\int_0^t  \norm{\zeta}_{L^\beta(\Omega)}\right)^{\frac 12}.  
\end{split}
\eeq 
Thus \eqref{errlinf} follows by \eqref{varthinf} and \eqref{Bzeta} applying with $\alpha=\infty$ . 

Using \eqref{keyb1} to bound $\norm{\vartheta}_{L^\infty(\Omega)}$ instead of \eqref{keyb} we obtain 
\beq\label{varthinf1}
\norm{p-p_h}_{L^\infty(\Omega)} \le \norm{\theta}_{L^\infty(\Omega)} +C\mathcal A^{\frac12} h^{-1} \left(\int_0^t  \norm{\zeta}_{L^{\delta}(\Omega)}^2\right)^{\frac 12}.  
\eeq 
This and and \eqref{Bzeta} applying with $\alpha=\infty$ give \eqref{varthinf1}. We finish the proof. 
\end{proof}
%=============================

Return to error estimate for vector gradient of pressure we have the following results  \begin{theorem}\label{suerr} Under the assumptions of Theorem \ref{mainres}. For any $0<t_0\le t\le T$  there is positive constants $C$ independent of $h$ such that

(i)
\beq\label{Bs-sh2a}
\begin{split}
\norm{(\s_h - \s)(t)}_{L^{\beta}(\Omega)}&\le C\mathcal A^{\frac {2\lambda+1}4 }\mathcal B^{\frac 12}h^{\frac{r+1}4}\sqrt{\int_0^t  \norm{\s(\tau)}_{r+1,\beta}d\tau}\\
&+ C\mathcal A^{\frac {\lambda+1}2} h^{\frac{r+1}2}\norm {\s(t)}_{r+1,\beta} .
\end{split}
\eeq
and
 \beq\label{Bu-uh2a}
\begin{split}
\norm{(\uu_h - \uu)(t)}_{L^{\beta}(\Omega)}&\le C\mathcal A^{\frac {2\lambda+1}4 }\mathcal B^{\frac 12}h^{\frac{r+1}4}\sqrt{\int_0^t  \norm{\s(\tau)}_{r+1,\beta}d\tau}\\
 &+ C\mathcal A^{\frac {\lambda+1}2 } h^{\frac{r+1}2}\norm {\s(t)}_{r+1,\beta} + Ch^{r+1} \norm{ \uu(t)}_{r+1,\beta}.
\end{split}
\eeq

 (ii) If $\s \in L^2(0,T; (W^{r+1,\delta}(\Omega))^d)$  then
\beq\label{Bs-sh2b}
\begin{split}
\norm{(\s_h - \s)(t)}_{L^{\beta}(\Omega)}&\le C\mathcal A^{\frac {3\lambda}4}\mathcal B^{\frac 12}h^{\frac{r+1}2}\sqrt{\int_0^t  \norm{\s(\tau)}_{r+1,\lambda}^2 d\tau}\\
&+ C\mathcal A h^{r+1}\norm {\s(t)}_{r+1,\delta} .
\end{split}
\eeq
and
 \beq\label{Bu-uh2b}
\begin{split}
\norm{(\uu_h - \uu)(t)}_{L^{\beta}(\Omega)}&\le C\mathcal A^{\frac {3\lambda}4 }\mathcal B^{\frac 12}h^{\frac{r+1}2}\sqrt{\int_0^t  \norm{\s(\tau)}_{r+1,\delta}^2 d\tau}\\
&+ C\mathcal Ah^{r+1}\norm {\s(t)}_{r+1,\delta} + Ch^{r+1} \norm{ \uu(t)}_{r+1,\beta}.
\end{split}
\eeq

\end{theorem}
%======================
\begin{proof}  

(i) Thank to \eqref{Mn1}, \eqref{eqerr} and $L^2$-projection,  
\beq\label{weightnorm}
\begin{split}
\omega\norm{\s_h - \s}^2_{L^{\beta}(\Omega)} &\le \left( K(|\s_h| )\s_h - K(|\s|) \s ,  \s_h - \s \right)\\
& =-(p_{h,t}- p_t, \vartheta)+\left( K(|\s_h| )\s_h - K(|\s|) \s , \zeta \right).
\end{split}
\eeq
 This, \eqref{mi1a} and \eqref{keyb} yield
\beq
\begin{split}
\omega\norm{\s_h - \s}^2_{L^{\beta}(\Omega)}&\le  C (\norm{p_{h,t}} +\norm{p_t} )\norm{\vartheta} + C\mathcal A\norm {  \zeta}_{L^{\beta}(\Omega)}\\
&\le C \mathcal A^{\frac 12} \mathcal B\left(\int_0^t  \norm{ \zeta}_{L^{\beta}(\Omega)}  d\tau\right)^{\frac 12} +C\mathcal A \norm {\zeta}_{L^{\beta}(\Omega)}.
\end{split}
\eeq
Thus 
\beq
\norm{\s_h - \s}^2_{L^{\beta}(\Omega)}\le C \mathcal A^{\frac 12}\mathcal B\omega^{-1} \left(\int_0^t  \norm{ \zeta}_{L^{\beta}(\Omega)}  d\tau\right)^{\frac 12} + C\omega^{-1}\mathcal A \norm {\zeta}_{L^{\beta}(\Omega)}.
\eeq
Due to \eqref{invomega}, we obtain 
\beqs\label{Bs2a}
\norm{\s_h - \s}^2_{L^{\beta}(\Omega)}\le C \mathcal A^{\lambda + \frac 12} \mathcal B \left(\int_0^t  \norm{ \zeta(\tau)}_{L^{\beta}(\Omega)}  d\tau\right)^{\frac 12} + C\mathcal A^{\lambda+1} \norm {\zeta(t)}_{L^{\beta}(\Omega)}.
\eeqs
Hence \eqref{Bs-sh2a} follows by \eqref{Bzeta}.

In \eqref{Erreq2}, let $\zz_h = \rho^{\beta-1}\in \tilde W_h$ and use H\"older's inequality we obtain 
\beq
\begin{aligned}
\norm{ \rho}_{L^{\beta}(\Omega)}^{\beta} &= - \left( K(|\s_h| )\s_h - K(|\s|) \s ,  \rho^{\beta-1} \right) \\
&\le C\left( |\s_h-\s|, \rho^{\beta-1}\right)\\
&\le C\norm{\s_h-\s}_{L^{\beta}(\Omega)}\norm{ \rho}^{\beta-1}_{L^{\beta}(\Omega)},  
\end{aligned}
\eeq
which gives 
\beqs
\norm{ \rho}_{L^{\beta}(\Omega)} \le C \norm{\s_h-\s}_{L^{\beta}(\Omega)}.  
\eeqs
Hence  
\beq\label{diffu-uh}
\begin{split}
\norm{\uu_h -\uu}_{L^{\beta}(\Omega)} &\le C\left( \norm{\rho}_{L^{\beta}(\Omega)} +\norm {\varrho}_{L^{\beta}(\Omega)} \right)\\
&\le C\left( \norm{\s_h-\s}_{L^{\beta}(\Omega)} +\norm {\varrho}_{L^{\beta}(\Omega)}\right). 
\end{split}
\eeq
Using \eqref{Bs-sh2a} and \eqref{Bvarrho} we obtain \eqref{Bu-uh2a}.

(ii)   We bound the right hand side of \eqref{weightnorm} by using Cauchy- Schwartz, triangle inequality and \eqref{mi1aa} to obtain 
\beqs
\begin{split}
\omega\norm{\s_h - \s}^2_{L^{\beta}(\Omega)}&\le  C (\norm{p_{h,t}} +\norm{p_t} )\norm{\vartheta} +\varep\norm{\s_h - \s}^2_{L^{\beta}(\Omega)} + C\mathcal \varep^{-1}\norm { \zeta}_{L^{\delta}(\Omega)}^2\\
&\le C \mathcal A^{\frac \lambda 2}\mathcal B  \left(\int_0^t  \norm{ \zeta}_{L^{\delta}(\Omega)}^2  d\tau\right)^{\frac 12} +\varep\norm{\s_h - \s}^2_{L^{\beta}(\Omega)} + C\mathcal \varep^{-1}\norm { \zeta}_{L^{\delta}(\Omega)}^2.
\end{split}
\eeqs
Then by \eqref{invomega},
\beqs
\begin{split}
(C_1\mathcal A^\lambda)^{-1}\norm{\s_h - \s}^2_{L^{\beta}(\Omega)}&\le C \mathcal A^{\frac \lambda 2}\mathcal B  \left(\int_0^t  \norm{ \zeta}_{L^{\delta}(\Omega)}^2  d\tau\right)^{\frac 12}\\
&\quad +\varep\norm{\s_h - \s}^2_{L^{\beta}(\Omega)} + C\mathcal \varep^{-1}\norm { \zeta}_{L^{\delta}(\Omega)}^2.
\end{split}
\eeqs
Selecting $\varep= \frac1{2C_1\mathcal A^\lambda}$ then
\beq
\norm{\s_h - \s}^2_{L^{\beta}(\Omega)}\le C \mathcal A^{\frac {3\lambda}2}\mathcal B \left(\int_0^t  \norm{ \zeta}_{L^{\delta}(\Omega) }^2  d\tau\right)^{\frac 12} +C\mathcal A^{2\lambda} \norm { \zeta}_{L^{\delta}(\Omega)}^2.
\eeq
This  and \eqref{Bzeta} lead to \eqref{Bs-sh2b}.
 
Inequality \eqref{Bu-uh2b} follows from \eqref{diffu-uh} and \eqref{Bs-sh2b}. 
 The proof is complete.    
\end{proof}    
    %======================= 
\subsection{Error analysis for fully discrete scheme} In analyzing this method, proceed in a similar fashion as for the semidiscrete method, we derive a error estimate for the fully discrete scheme. 
Let
$p^n(\cdot) = p(\cdot,t_n)$, $\vv^n(\cdot) = \vv(\cdot,t_n)$ and $\uu^n(\cdot) = \uu(\cdot,t_n)$ be the
true solution evaluated at the discrete time levels.  We will also
denote $\pi p^n \in W_h$, $\pi \s^n \in \tilde W_h$ and $\Pi \uu^n \in V_h$ to be the projections
of the true solutions at the discrete time levels.  
%===========================

We rewrite \eqref{weakform} with $t=t_n$. Using the definitions of projections and assumption that $\nabla\cdot V_h \subset W_h $, standard
manipulations show that the true solution satisfies the discrete equation
\begin{subequations}\label{fulprjsys}
\begin{align}
\label{fulprj1}& \left( \frac{\pi p^n - \pi p^{n-1}}{\Delta t }, w_h\right)  +  \left(\nabla\cdot \Pi\uu^n, w_h\right) =(f^n, w_h) +(\epsilon^n,w_h) , &\forall w_h\in W_h\\
\label{fulprj2}& (\Pi\uu^n, \zz_h)  + ( K(|\s^n| )\s^n ,\zz_h)=0, &\forall \zz_h\in \tilde W_h,\\
 \label{fulprj3} &(\pi\s^n,\vv_h)  + ( \pi p^n , \nabla\cdot  \vv_h )=0, &\forall \vv_h\in V_h,
\end{align}
\end{subequations}  
where $\epsilon^n$ is the time truncation error of order $\Delta t$. 
%==================================
\begin{theorem}\label{fulErr} Assume $(\bar p^0,\uu^0,\s^0 )\in W\times V\times \tilde W$ and $(\bar p_h^0,\uu^0_h,\s_h^0 )\in W_h\times V_h\times \tilde W_h$. Let $(p, \uu,\s)$ solve problem \eqref{weakform} and $(p_h^n, \uu_h^n,\s_h^n)$ solve the fully  discrete mixed finite element approximation \eqref{fullydiscreteform} for each time step $n$, $n=1\ldots, N$.  
There exists a positive constant $C$ independent of $h$ and $\Delta t$  such that if  the $\Delta t$ is sufficiently small then 
\beq\label{fulerr2}
\norm{  p_h^m -  p^m} \le C(h^{\frac{r+1}2}+\Delta t) 
\eeq
for $m =1,\dots, N.$

 Moreover if $\s^n\in \left(W^{r+1,\delta}(\Omega)\right)^d$ for  $n=1,\dots, N$ then  
\beq\label{fulerr2a}
\norm{  p_h^m -  p^m} \le C(h^{r+1}+\Delta t) 
\eeq
for $m =1,\dots, N.$
\end{theorem}
\begin{proof} 
Subtracting \eqref{fullydiscreteform} from \eqref{fulprjsys}, in the resultants using $w_h= \vartheta^n, \zz_h = \eta^n,\vv_h= \rho^n$  we obtain 
\begin{subequations}
\begin{align}
&\label{b1} \left( \frac{\vartheta^n- \vartheta^{n-1}}{\Delta t}, \vartheta^n \right) +  \left(\nabla\cdot \rho^n, \vartheta^n\right) =(\epsilon^n,\vartheta^n), \\
&\label{b2} (\rho^n, \eta^n )  + \left( K(|\s^n_h| )\s^n_h -K(|\s^n|) \s^n ,  \eta^n \right)=0,\\ 
&\label{b3} (\eta^n,\rho^n)  + ( \vartheta^n , \nabla\cdot \rho^n )= 0.
\end{align}
\end{subequations} 
Combining \eqref{b1}--\eqref{b3} gives 
\beqs
\norm{ \vartheta^n}^2 + \Delta t \left( K(|\s^n_h| )\s^n_h -K(|\s^n|) \s^n ,  \eta^n \right) = ( \vartheta^n, \vartheta^{n-1}  )+\Delta t(\epsilon^n,\vartheta^n).
\eeqs
This equation is equivalent to  
\beq\label{c11}
\begin{aligned}
&\norm{ \vartheta^n}^2 + \Delta t \left( K(|\s^n_h| )\s^n_h -K(|\s^n|) \s^n ,  \s^n_h-\s^n \right) \\
&\quad\quad= (\vartheta^n, \vartheta^{n-1}  )+ \Delta t\Big\{ \left( K(|\s^n_h| )\s^n_h -K(|\s^n|) \s^n ,  \zeta^n \right)+(\epsilon^n,\vartheta^n)\Big\}.
\end{aligned}
\eeq
The second term of \eqref{c11}, using \eqref{Mono}, is bounded  : 
\beq\label{Wssh}
\left( K(|\s^n_h| )\s^n_h -K(|\s^n|) \s^n ,  \s^n_h-\s^n \right) \ge C\omega^n  \norm{\s^n_h -\s^n}^2_{L^{\beta}(\Omega)},
\eeq
where $\omega^n =\omega (t_n).$

The right hand side of \eqref{c11} using Cauchy's inequality and \eqref{mi1a} and \eqref{invome} give   
\begin{multline} \label{RHSc11}
(\vartheta^n, \vartheta^{n-1}  )+ \Delta t\left( \left( K(|\s^n_h| )\s^n_h -K(|\s^n|) \s^n ,  \zeta^n \right)+(\epsilon^n,\vartheta^n)\right)\\
\le \frac 12 \left(\norm{\vartheta^n}^2 +\norm{\vartheta^{n-1}}^2   \right)+
\Delta t\Big\{C\mathcal A\norm{\zeta^n}_{L^{\beta}(\Omega)} +\frac12\left(\norm{\vartheta^n}^2 +\norm{\epsilon^n}^2\right) \Big\}.
\end{multline}
It follows from \eqref{c11}, \eqref{Wssh} and \eqref{RHSc11} that   
\beqs
\norm{ \vartheta^n}^2 - \norm{\vartheta^{n-1}}^2+ C\Delta t\omega^n \norm{\s^n_h-\s^n}^2_{L^{\beta}(\Omega)} \le\Delta t \norm{ \vartheta^n}^2 + C \Delta t \left(\mathcal A \norm{\zeta^n}_{L^{\beta}(\Omega)}+\norm{\epsilon^n}^2\right). 
\eeqs
Summing  over $n$
\begin{align*}
(1-\Delta t)\norm{ \vartheta^m}^2 &+ C \sum_{n=1}^m \Delta t  \omega^n\norm{\s^n_h-\s^n}^2_{L^{\beta}(\Omega)}\\
 &\le\sum_{n=1}^{m-1} \Delta t \norm{ \vartheta^n}^2  
+ C \sum_{n=1}^m \Delta t \left(\mathcal A  \norm{\zeta^n}_{L^{\beta}(\Omega)}+\norm{\epsilon^n}^2\right)
\end{align*}
for some $m=2,\ldots, N.$

Dropping the nonnegative term of the left hand side, using Gronwall's lemma, we obtain   
 \beq\label{thetam}
\norm{\vartheta^m}^2 \le C\sum_{n=1}^m\Delta t \left(\mathcal A \norm{\zeta^n}_{L^{\beta}(\Omega)}+\norm{\epsilon^n}^2\right). 
\eeq
The triangle inequality gives 
 \beqs
 \norm{ p^m_h- p^m}^2 \le C\mathcal A\sum_{n=1}^m\Delta t  
\norm{\zeta^n}_{L^{\beta}(\Omega)}
+ \norm{\theta^m}^2+C(\Delta t)^2. 
\eeqs
This and properties of projections lead to \eqref{fulerr2} true.

(ii) We prove the superconvergence by estimate the right hand side of \eqref{c11} using Cauchy's inequality, \eqref{mi1aa} and \eqref{invome} to obtain    
\begin{multline}\label{RHSc}
(\vartheta^n, \vartheta^{n-1}  )+ \Delta t\left( \left( K(|\s^n_h| )\s^n_h -K(|\s^n|) \s^n ,  \zeta^n \right)+(\epsilon^n,\vartheta^n)\right)\le \frac 12 \left(\norm{\vartheta^n}^2 +\norm{\vartheta^{n-1}}^2   \right)
\\+\Delta t\left\{\varep\omega^n\norm{\s^n-\s_h^n}_{L^{\beta}(\Omega)}^2 +C_1(\varep\omega^n) ^{-1}\norm{\zeta^n}_{L^{\delta}(\Omega) }^2 +\frac12\left(\norm{\vartheta^n}^2 +\norm{\epsilon^n}^2\right) \right\}.
\end{multline}
Now we combine \eqref{Wssh}, \eqref{c11} and \eqref{RHSc} to have   
\begin{align*}
\norm{ \vartheta^n}^2 - \norm{\vartheta^{n-1}}^2+ C\Delta t\omega^n \norm{\s^n_h-\s^n}^2_{L^{\beta}(\Omega)} \le\Delta t \norm{ \vartheta^n}^2 + 2\varep \Delta t\omega^n
\norm{\s^n_h-\s^n}^2_{L^{\beta}(\Omega)}\\
+ C_1 \Delta t \left((\varep\omega^n)^{-1} \norm{\zeta^n}_{L^{\delta}(\Omega) }^2+\norm{\epsilon^n}^2\right). 
\end{align*}
Selecting $\varep =C/4$ we obtain  
\begin{align*}
\norm{ \vartheta^n}^2 - \norm{\vartheta^{n-1}}^2&+ C\Delta t\omega^n \norm{\s^n_h-\s^n}^2_{L^{\beta}(\Omega)}\\
 &\le\Delta t \norm{ \vartheta^n}^2 + C_2 \Delta t \left((\omega^n)^{-1} \norm{\zeta^n}_{L^{\delta}(\Omega) }^2+\norm{\epsilon^n}^2\right)\\
&\le\Delta t \norm{ \vartheta^n}^2 + C_2 \Delta t \left(\mathcal A^\lambda \norm{\zeta^n}_{L^{\delta}(\Omega) }^2+\norm{\epsilon^n}^2\right). 
\end{align*}
Now we drop the the nonnegative term in the left hand side in above inequality, sum over $n$ and use Gronwall's inequality to find that
\beqs
\norm{\vartheta^m}^2 \le C\sum_{n=1}^m\Delta t \left(\mathcal A^\lambda \norm{\zeta^n}_{L^{\delta}(\Omega) }^2+\norm{\epsilon^n}^2\right). 
\eeqs
Again using triangle inequality, properties of projections we obtain \eqref{fulerr2a}.
\end{proof}
%=====================
\begin{theorem}\label{Derr}
Under the assumptions of Theorem \ref{fulErr}. There exists a positive constant $C$ independent of $h$ and $\Delta t$  such that if  the $\Delta t$ is sufficiently small then 
 
\beq\label{Dsuh0}
\norm{\s^m_h - \s^m}_{L^{\beta}(\Omega) }+\norm{\uu^m_h - \uu^m}_{L^{\beta}(\Omega)}\le C( h^{\frac{r+1}4}+ \sqrt{\Delta t} )
\eeq
for all $m=1,\dots, N$. 

Furthermore if $\s^n \in (W^{r+1,\delta}(\Omega))^d$ for all $n=1,\dots, N$ then  
\beq\label{Dsuh1}
\norm{\s^m_h - \s^m}_{L^{\beta}(\Omega)}+\norm{\uu^m_h - \uu^m}_{L^{\beta}(\Omega)}\le C( h^{\frac{r+1}2}+ \sqrt{\Delta t} ) 
\eeq
for all $m=1,\dots, N.$
\end{theorem}
%======================
\begin{proof}  Recall that the true solution satisfies the discrete equations 
\begin{subequations}\label{tem1}
\begin{align}
\label{te1}& \left( p_t^n, w_h\right)  +  \left(\nabla\cdot \Pi\uu^n, w_h\right) =(f^n, w_h)  , &&\forall w_h\in W_h\\
\label{te2}& (\Pi\uu^n, \zz_h)  + ( K(|\s^n| )\s^n ,\zz_h)=0, &&\forall \zz_h\in \tilde W_h,\\
 \label{te3} &(\pi\s^n,\vv_h)  + ( \pi p^n , \nabla\cdot  \vv_h )=0, &&\forall \vv_h\in V_h,
\end{align}
\end{subequations} 
Subtracting \eqref{fullydiscreteform} from \eqref{tem1}, 
%we obtain 
%\begin{subequations}\label{tem3}
%\begin{align}
%&\label{te31}    \left( \frac{ p_h^n -  p_h^{n-1}}{\Delta t } -p_t^n, w_h\right) + \left(\nabla \cdot \rho^n, w_h\right) =0,   &&\forall w_h\in W_h, \\
%&\label{te32}     (\rho^n, \zz_h)  + ( K(|\s_h^n| )\s_h^n - K(|\s^n| )\s^n ,\zz_h )=0, &&\forall \zz_h\in {\tilde W}_h,\\
% &\label{te33}     (\eta^n,{\bf v_h})  + ( \vartheta^n , \nabla\cdot \vv_h )=0, &&\forall \vv_h\in V_h.
%\end{align}
%\end{subequations} 
choosing $w_h=\vartheta^n$, $\zz_h =\eta^n$, $\vv_h =\rho^n$, we obtain  
\begin{subequations}\label{tem4}
\begin{align}
&\label{te41}    \left( \frac{ p_h^n -  p_h^{n-1}}{\Delta t } -p_t^n, \vartheta^n\right) + \left(\nabla \cdot \rho^n, \vartheta^n\right) =0, \\
&\label{te42}     (\rho^n, \eta^n)  + ( K(|\s_h^n| )\s_h^n - K(|\s^n| )\s^n ,\eta^n )=0,\\
 &\label{te43}     (\eta^n,\rho^n)  + ( \vartheta^n , \nabla\cdot \rho^n )=0.
\end{align}
\end{subequations} 
Above equations yield   
\beq\label{erq}
\left( \frac{ p_h^n -  p_h^{n-1}}{\Delta t } -p_t^n, \vartheta^n\right) + \left( K(|\s_h^n| )\s_h^n - K(|\s^n| )\s^n ,\eta^n \right)=0.
\eeq
We use \eqref{Mn1}, \eqref{erq} to find that     
\beq\label{pior-est}
\begin{split}
\omega^n\norm{\s^n_h - \s^n}^2_{L^{\beta}(\Omega)} &\le \left( K(|\s^n_h| )\s^n_h - K(|\s^n|) \s^n ,  \s^n_h - \s^n \right)\\
&=\left( K(|\s^n_h| )\s^n_h - K(|\s^n|) \s^n ,  \eta^n \right)+ \left( K(|\s^n_h| )\s^n_h - K(|\s^n|) \s^n ,   \zeta^n \right)\\
& =-\left(\frac{p_{h}^n - p_h^{n-1}}{\Delta t} - p_t^n, \vartheta^n\right)+\left( K(|\s^n_h| )\s^n_h - K(|\s^n|) \s^n , \zeta^n \right).
\end{split}
\eeq
Due to \eqref{mi1a}, Cauchy-Schwartz and triangle inequality, one has  
\beqs
\omega^n\norm{\s^n_h - \s^n}^2_{L^{\beta}(\Omega)}\le C \left((\Delta t)^{-1} \norm{ p_{h}^n - p_h^{n-1}} +\norm{p_t^n} \right)\norm{\vartheta^n} + \mathcal A\norm {  \zeta^n}_{L^{\beta}(\Omega)}. 
\eeqs
Using the fact that
\begin{align*}
(\Delta t)^{-1}\norm{p_{h}^n - p_h^{n-1}} &=(\Delta t)^{-1}\norm{ \int_{t_{n-1}}^{t_n} p_{h,t} dt}\\
& \le (\Delta t)^{-1} \int_{t_{n-1}}^{t_n} \norm{p_{h,t}} dt\le \sup_{[T/N,T]}\norm{p_{h,t}}\le \mathcal B, 
\end{align*}
and 
\beqs 
\norm{p_t^n} \le \sup_{[T/N,T]}\norm{p_t}\le \mathcal B, 
\eeqs 
we obtain   
\beqs
\norm{\s^n_h - \s^n}^2_{L^{\beta}(\Omega)}
\le C \mathcal B(\omega^n)^{-1}\norm{\vartheta^n} + \mathcal A(\omega^n)^{-1}\norm {  \zeta^n}_{L^{\beta}(\Omega)}.
\eeqs
It follows from \eqref{thetam} and \eqref{invome} that
\beq\label{fDsh}
\begin{split}
\norm{\s^n_h - \s^n}^2_{L^{\beta}(\Omega)}
&\le C \mathcal B (\omega^n)^{-1}\left\{\sum_{i=1}^n\Delta t \left(\mathcal A \norm{\zeta^i}_{L^{\beta}(\Omega)}+\norm{\epsilon^i}^2\right)\right\}^{\frac 12} + \mathcal A (\omega^n)^{-1}\norm {\zeta^n}_{L^{\beta}(\Omega)}\\
&\le C  \mathcal A^\lambda \mathcal B \left\{ \left(\mathcal A\sum_{i=1}^n\Delta t  \norm{\zeta^i}_{L^{\beta}(\Omega)}\right)^{\frac 12}+\Delta t    \right\} + C\mathcal A^{\lambda+1} \norm {\zeta^n}_{L^{\beta}(\Omega)}.
\end{split}
\eeq
Thus 
\beq\label{Dsh}
\begin{split}
\norm{\s^n_h - \s^n}_{L^{\beta}(\Omega)}&\le C\mathcal A^{\frac \lambda 2+\frac 14 }\mathcal B^{\frac 12} h^{\frac{r+1}4}\left( \sum_{i=1}^n  \Delta t\norm{\s^i}_{r+1,\beta}d\tau\right)^{\frac 14}\\
& \quad+ C\mathcal A^{\frac \lambda 2+\frac 12 } h^{\frac{r+1}2}\norm {\s^n}_{r+1,\beta}^{\frac 12}
+ C \mathcal A^\lambda \mathcal B\sqrt{\Delta t}.
\end{split}
\eeq

The triangle inequality gives
\beqs
\norm{\uu^n_h -\uu^n}_{L^{\beta}(\Omega)}\le C( \norm{\rho^n}_{L^{\beta}(\Omega)} +\norm {\varrho^n}_{L^{\beta}(\Omega)}). 
\eeqs
Subtracting \eqref{fully2} from \eqref{te2} and using $\zz_h =(\rho^n) ^{\beta-1} $ we have equation  
 \beqs
 \left(\rho^n, (\rho^n)^{\beta-1}\right) + \left( K(|\s^n_h| )\s^n_h -K(|\s^n|) \s^n ,  (\rho^n)^{\beta-1} \right)=0.
 \eeqs 
 Then according Cauchy-Schwartz inequality and Proposition \ref{conts}, 
 \beqs
 \norm{ \rho^n}_{L^{\beta}(\Omega)} \le C\norm{\s^n_h -\s^n}_{L^{\beta}(\Omega)}.
 \eeqs
 Hence 
 \beq\label{ubound}
\norm{\uu^n_h -\uu^n}_{L^{\beta}(\Omega)}\le C( \norm{\s^n_h-\s^n}_{L^{\beta}(\Omega)} +\norm {\varrho^n}_{L^{\beta}(\Omega)}). 
\eeq
Using \eqref{Dsh} and \eqref{Bvarrho} yield  
\beq\label{Duh}
\begin{split}
\norm{\uu^n_h - \uu^n}_{L^{\beta}(\Omega)}&\le C\mathcal A^{\frac \lambda 2+\frac 14 }\mathcal B^{\frac 1 2}h^{\frac{r+1}4}\left( \sum_{i=1}^n  \Delta t\norm{\s^i}_{r+1,\beta}d\tau\right)^{\frac 12}\\
& \quad+ C\mathcal A^{\frac \lambda 2+\frac 12 } h^{\frac{r+1}2}\norm {\s^n}_{r+1,\beta}^{\frac 12}+Ch^{r+1}\norm {\uu^n}_{r+1,\beta}
+ C\mathcal A^\lambda  \mathcal B \sqrt{\Delta t}.
\end{split}
\eeq
Therefore \eqref{Dsuh0} follows from \eqref{Dsh} and \eqref{Duh}. 

(ii) Thank to the regularity of solution we bound the right hand side of \eqref{pior-est} using \eqref {mi1aa} instead of Cauchy-Schwartz inequality to obtain 
\beqs
\omega^n\norm{\s^n_h - \s^n}^2_{L^{\beta}(\Omega)}
\le C \mathcal B\norm{\vartheta^n} +  \varep \omega^n \norm{\s^n_h - \s^n}^2_{L^{\beta}(\Omega)}  +  C(\varep\omega^n )^{-1}\norm {  \zeta^n}_{L^{\delta}(\Omega)}^2.
\eeqs
or 
\beqs
\norm{\s^n_h - \s^n}^2_{L^{\beta}(\Omega)}
\le C \mathcal B(\omega^n)^{-1}\norm{\vartheta^n} +  \varep \norm{\s^n_h - \s^n}^2_{L^{\beta}(\Omega)}  +  C\varep^{-1}(\omega^n )^{-2}\norm {  \zeta^n}_{L^{\delta}(\Omega)}^2.
\eeqs
Selecting $\varep=\frac 12$, it follows from \eqref{thetam} and \eqref{invome} that 
\beqs
\norm{\s^n_h - \s^n}^2_{L^{\beta}(\Omega)}\le C \mathcal A^\lambda\mathcal B  \left\{\mathcal A^{\frac \lambda 2} \left(\sum_{i=1}^n\Delta t  \norm{\zeta^i}_{L^{\delta}(\Omega)}^2\right)^{\frac 12}+\Delta t    \right\} + C\mathcal A^{2\lambda} \norm {\zeta^n}_{L^{\delta}(\Omega)}^2.
\eeqs
Thus
\beqs
\begin{split}
\norm{\s^n_h - \s^n}_{L^{\beta}(\Omega)}&\le C\mathcal A^{\frac{3\lambda}4 }\mathcal B^{\frac 12}h^{\frac{r+1}4}\left( \sum_{i=1}^n  \Delta t\norm{\s^i}_{r+1,\delta}^2d\tau\right)^{\frac 14}\\
& \quad+ C\mathcal A^\lambda h^{\frac{r+1}2}\norm {\s^n}_{r+1,\delta}
+ C\mathcal A^{\frac \lambda 2} \mathcal B\sqrt{\Delta t}.
\end{split}
\eeqs
This and \eqref{ubound} give us \eqref{fulerr2a}. We finish the proof . 
\end{proof}

%==============================
\section{Numerical results}
In this section, we give a simple numerical result illustrating the convergence
theory. We test the convergence of our method with the Forchheimer two term law. 
For simplicity, consider $g(s)=1+s$. Equation \eqref{eq4} $sg(s)=\xi,$ $s\ge0$ gives
 $ s= \frac {-1 +\sqrt{1+4\xi}}{2}$ and hence 
$$
K(\xi) =\frac1{g(s(\xi)) } =\frac {2}{1+\sqrt{1+4\xi}}.
$$ 
%For simplicity, we take $A=B=1$, in this case $ K(\xi) =\frac {2}{1+\sqrt{1+4\xi}}. $ 
Since we analyze a first order time discretization, we consider the lowest order mixed method. Here we use the lowest order Raviart-Thomas mixed finite element on the unit square in two dimensions. The chosen analytical solution is 
\begin{align*}
p(x,t)&=e^{-5t}\left[\frac12 (x^2_1 +x^2_2)-\frac 13(x^3_1+x^3_2) \right],\\
\s(x,t) &=\nabla p=e^{-5t}( x_1(1-x_1), x_2(1-x_2) ), \\
\uu(x,t) &=K(|\s|)\s=\frac {2\s(x,t)}{1+\sqrt{1+4|\s(x,t)|}}
\end{align*}
 for all  $x\in\Omega,t\in [0,1]$ where $x=(x_1,x_2), \Omega=[0,1]^2$. The forcing term $f$ is determined accordingly to the analytical solution by equation $p_t - \nabla \cdot\uu = f$. Explicitly,  
\beqs
\begin{split}
 f(x,t)&=\frac12 (x^2_1 +x^2_2)-\frac 13(x^3_1+x^3_2)-\frac {4e^{-5t}(1-x_1-x_2)}{1+\sqrt{1+4|\s|}}\\
&+\frac{4e^{-15t}}{|\s| (1+\sqrt{1+4|\s|})^2 \sqrt{1+4|\s|}}\Big[  x_1^2(1-x_1)^2(1-2x_1)+ x_2^2(1-x_2)^2(1-2x_2) \Big].
\end{split}
\eeqs   

We used FEniCS \cite{fenics} to perform our numerical simulations. We divide the unit square into an $ N\times N$ mesh of squares, each then subdivide into two right triangles using the \textsf{UnitSquareMesh} class in FEniCS. For each mesh, we solve the generalized Forchheimer equation numerically. The error control in each nonlinear solve is $\varep =10^{-6}$.  Our problem is solved at each time level start at $t=0$ until final time $T = 1$. At this time, we measured the $L^2$-errors of pressure and $L^{\beta}$-errors of gradient of pressure and velocity. Here $\beta = 2 - a = 2- \frac{{\rm deg} (g)}{{\rm deg} ( g) + 1}=\frac 32$. The numerical results are listed as the following table.   
 
 \vspace{0.3cm}  
\begin{center}
\begin{tabular}{c|c| c|c| c|c|c}
%\hline
N    & $\norm{p-p_h}$ & Rates & $\norm{\s-\s_h}_{L^{\beta}(\Omega)}$ &  Rates & $\norm{\uu-\uu}_{L^{\beta}(\Omega)}$ &  Rates \\
\hline
4& 1.965e-01 	&-& 2.505e-01&-&2.436e-01&-\\
8&1.011e-01	&1.94&2.523e-01&0.99&2.504e-01&0.97\\
16&5.081e-02&1.98&2.525e-01&0.99&2.517e-01&0.99 \\
32&2.542e-02&1.99&2.525e-01&1.00&2.519e-01&0.99\\
64&1.270e-02&2.00&2.524e-01&1.00&2.519e-01&1.00\\
128&6.351e-03&1.99&2.523e-01&1.00&2.519e-01&1.00\\
256&3.175e-03&2.00&2.521e-01&1.00&2.519e-01&1.00\\
\hline
\end{tabular}

\vspace{0.3cm}
Table 1. {\it Convergence study for generalized Forchheimer equation with zero flux on the boundary in 2D.}

\end{center}

\bibliographystyle{siam}
%\bibliography{Fpapers}
\def\cprime{$'$} \def\cprime{$'$}

\end{document}